\documentclass[a4paper,12pt]{amsart}

\usepackage{amsthm}
\usepackage{amsmath}
\usepackage{amsfonts}
\usepackage{amssymb}
\usepackage{latexsym}
\usepackage{times}
\usepackage{graphicx}
\usepackage{yhmath}

\theoremstyle{definition}
\newtheorem{defn}{\indent\bf Definition}
\newtheorem{rem}[defn]{\indent\bf Remark}

\newtheorem{obs}[defn]{\indent\bf Observation}
\newtheorem{prob}[defn]{\indent\bf Problem}

\theoremstyle{plain}

\newtheorem{prop}[defn]{\indent\bf Proposition}
\newtheorem{thm}[defn]{\indent\bf Theorem}
\newtheorem{cor}[defn]{\indent\bf Corollary}

\begin{document}

\title[On unbounded, non-trivial Hochschild cohomology]{On unbounded, non-trivial Hochschild cohomology in finite von Neumann algebras\\
and higher order Berezin's quantization}
\author[Florin R\u adulescu]{Florin R\u adulescu${}^*$
  \\ \\ 
Dipartimento di Matematica\\ Universita degli Studi di Roma ``Tor Vergata''}
\dedicatory{Dedicated to Professor \c Serban Valentin Str\u atil\u a, on the occasion of his 70th anniversary}
\maketitle 

\thispagestyle{empty}

\def\tilde{\widetilde}
\def\a{\alpha}
\def\T{\theta}
\def\PSL{\mathop{\rm PSL}\nolimits}
\def\SL{\mathop{\rm SL}\nolimits}
\def\PGL{\mathop{\rm PGL}}
\def\Per{\mathop{\rm Per}}
\def\GL{\mathop{\rm GL}}
\def\Out{\mathop{\rm Out}}
\def\Int{\mathop{\rm Int}}
\def\Aut{\mathop{\rm Aut}}
\def\ind{\mathop{\rm ind}}
\def\Im{\mathop{\rm Im}}
\def\Tr{\mathop{\rm Tr}}
\def\Re{\mathop{\rm Re}}
\def\Dom{\mathop{\rm Dom}}
\def\card{\mathop{\rm card}}
\def\d{{\rm d}}
\def\Z{\mathbb Z}
\def\R{\mathbb R}
\def\bC{\mathbb C}
\def\bD{\mathbb D}
\def\cR{{\mathcal R}}
\def\tPsi{{\tilde{\Psi}}}
\def\Q{\mathbb Q}
\def\N{\mathbb N}
\def\C{\mathbb C}
\def\bH{\mathbb H}
\def\Y{{\mathcal Y}}
\def\E{{\mathcal E}}
\def\L{{\mathcal L}}
\def\G{{\mathcal G}}
\def\M{{\mathcal M}}
\def\U{{\mathcal U}}
\def\F{{\mathcal F}}
\def\H{{\mathcal H}}
\def\I{{\mathcal I}}
\def\A{{\mathcal A}}
\def\S{{\mathcal S}}
\def\O{{\mathcal O}}
\def\V{{\mathcal V}}
\def\D{{\mathcal D}}
\def\B{{\mathcal B}}
\def\K{{\mathcal K}}
\def\cC{{\mathcal C}}
\def\cR{{\mathcal R}}
\def\T{{\mathcal T}}
\def\cX{{\mathcal X}}
\def\ptimes{\mathop{\boxtimes}\limits}
\def\potimes{\mathop{\otimes}\limits}

\begin{abstract}
We introduce a class of densely defined, unbounded, 2-Hoch\-schild cocycles ([PT])
on finite von Neumann algebras $M$. Our cocycles admit a coboundary,  determined by an unbounded operator on the standard 
Hilbert space associated to the von Neumann algebra $M$.
For the cocycles associated to the $\Gamma$-equivariant deformation ([Ra]) of
the upper halfplane $(\Gamma=\PSL_2(\Z))$, the ``imaginary" part of the coboundary ope\-rator is
a cohomological obstruction -- in the sense that it can not be removed by a "large class" of closable derivations, with
non-trivial real part, that  have a joint core domain, with the given coboundary.

As a byproduct,
we prove a strengthening of the non-triviality of the Euler cocycle in the bounded cohomology ([Br]) $H^2_{\rm bound}(\Gamma,\Z)$.

\end{abstract}

\renewcommand{\thefootnote}{}
\footnotetext{${}^*$ Member of the Institute of  Mathematics ``S. Stoilow" of the Romanian Academy}
\footnotetext{${}^*$  
Supported in part by PRIN-MIUR, and PN-II-ID-PCE-2012-4-0201
}

\section*{Introduction}

In this paper we analyze from an abstract point a view a cohomological obstruction that appears
in the study of the $\Gamma$-equivariant Berezin's deformation quantization of the upper half
plane. In this paper $\Gamma=\PSL_2(\Z)$, or a congruence subgroup.

In the case of the $\Gamma$-equivariant Berezin's quantization of the upper halfplane $\bH\subseteq \C$ considered in [Ra1], [Ra2], (see also [NN], [Bi]), the deformation consists in a family of type $II_1$ von Neumann  algebras
$(\A_t)_{t>1}$ indexed by  the parameter   $t\in (1,\infty)$, and a symbol map which we describe below. Each algebra $(\A_t)_{t>1}$   is embedded in the bounded operators acting on the  Hilbert space $H_t $,  consisting of analytic functions on the upper halfplane, square integrable with respect to a measure depending on $t>1$.

These Hilbert spaces are endowed with projective  unitary representations  
 $\pi_t:\PSL_2(\R)\rightarrow B(H_t), t>1,$
from the extended discrete series (see [Pu]), of projective unitary representations of $\PSL_2(\R)$. In this representation,  the algebra $\A_t$ is the commutant  algebra
$\{\pi_t(\Gamma)\}'$. The algebras $\A_t$  are finite von Neumann algebras (factors, i.e they have trivial centers),  as shown in [GHJ], [Ra1].

As observed in [GHJ], the automorphic forms correspond to intertwiners for the representations $\pi_t$, when restricted to $\Gamma$. Using
the canonical branch for the logarithm of the Dedekind $\Delta$ function,
 the multiplication operators by the powers $\Delta^{\varepsilon}$, $\varepsilon>0$, of $\Delta$, give intertwining operators for  the representations $\pi_{t}\vert\Gamma$, $\pi_{t_+\varepsilon}\vert\Gamma$ acting on the Hilbert spaces  $H_t$, and respectively $H_{t+\varepsilon}$. We denote these
injective, linear, bounded, intertwining multiplication operators by $S_{t+\varepsilon,t}$. These operators are compared 
with the canonical inclusions $I_{t+\varepsilon,t}$, $H_t\rightarrow H_{t+\varepsilon}$.

Out of this data (see  [Ra2]), one obtains a family of completely positive maps
$\Psi_{s,t}:\A_t\rightarrow\A_s, s\geq t>1$, the Berezin symbol map.

 Concomitantly, using the intertwiners described above, one obtains   a family $$\beta_{s,t}:\A_t\rightarrow\A_s, \beta_{s,t}(a)= S_{s,t}a(S_{s,t})^\ast,  s\geq t>1, a\in \A_t$$ 
of completely positive maps. Consequently, the linear map on $\A_t$ with values in $\A_s$, defined by the formula:
$$
\alpha_{s,t}(a)=\beta_{s,t}(1)^{-1/2}\beta_{s,t}(a)\beta_{s,t}^{-1/2}(1),\ \  a\in \A_s
$$ 
is an isomorphism from $\A_t$ onto $E_t^s \A_sE_t^s$, where $E_t^s$
is a projection, of trace $\frac{\chi_t}{\chi_s}$, in $\A_s$. More precisely $E_t^s$ is the orthogonal projection on the closure of the image of  $S_{s,t}$. Note that, 
$\beta_{s,t}(1)^{1/2}$ is the absolute value of the intertwining operator  $S_{s,t}$. Clearly,
$\alpha_{s,t}$ is an isomorphism onto its image.
Here $\chi_t, t>1,$ is a linear,  increasing function, with positive values, depending on
 $t$, related to the Plancherel dimension (coefficient) of $\pi_t$ (see e.g. [Pu]). 

In this paper we initiate the axiomatization of the properties of the Berezin  deformation quantization, properties which were used in ([Ra2])  to construct a cohomological data for the deformation. This cohomological data consists of a family of unbounded Hochschild 2-cocycles $(c_t)_{t>1}$ (see [PT], [Co], [KR], [Ra2] for definitions)  defined on  dense unital $\ast-$ subalgebras of $\A_t$    (see  [Ra2]).

 The cocycles $(c_t)_{t>1}$  are associated to the deformation; they represent the obstruction to construct isomorphisms between the algebras corresponding to different values of the deformation parameter $t$. Formally, this 2-cocycle is obtained by differentiating the multiplication operation for operators with fixed symbols.

We prove in Theorem \ref{ind}  that the ``imaginary part''  $c_t^0$ of the unbounded Hochschild cocycle $c_t$ (see Definition \ref{im}, and see also [CC]),   associated
to this deformation data, is implemented, for a fixed $t$,  by the ``infinitesimal generator'', at $t$, of the family $(\alpha_{s,t})_s$.

 We prove that this generator  is of the form
$Z_t=\lambda_t1+iY_t$, where $Y_t$ is a symmetric, unbounded operator, acting on the standard Hilbert space $L^2(\A_t)$ associated to the algebra $\A_t$ and to its unique trace, 
 with the additional property that $iY_t^*1=-\lambda_t1,$ where $\lambda_t=\frac12\frac{\chi'_t}{\chi_t}>0$.   This clearly prevents the identity element $1$ to belong to the domain of $Y_t$. Thus $(c_t^0,\lambda_t1+iY_t)$
is an invariant "characterizing" the algebra $\A_t$ and which depends on the domain for the unbounded operators considered. 

In terms of  the space of Berezin's symbols for operators in $B(H_t)$ this construction is described as follows:
an operator $A\in B(H_t)$ is represented by a kernel $k_A=k_A(z,\overline{\eta})$
on $\bH\times\bH$, analytic in the first variable and antianalytic in the second,
so that
$$
Af(z)=\chi_t\int_\bH\frac{k_A(z,\overline{\eta})}{(z-\overline{\eta})^t}f(\eta)
d\nu_t(\eta),\quad z,\eta\in\bH,
$$
for $f$ in $H_t=H^2(\bH,d\nu_t)$, $t>1$. Here, for $t\geq 0$ we let $d\nu_t$ be the measure  $(\Im z)^{t-2}dzd\overline{z}$ on $\bH$. Moreover, we let $\chi_t=\frac{t-1}{\pi}$. To indicate the antianalytic dependence in the second variable, we put the conjugation symbol on the second variable.
We identify $A$ with its reproducing kernel $k_A$.

We obtain the following formula representing the kernel of the product in $B(H_t)$ of the operators defined by  kernels $k,l$
$$
(k\ast_t l){(z,\overline{\zeta})}=
\chi_t\int_{\bH}k(z,\overline{\eta})l(\eta,\overline{\zeta})
[z,\overline{\eta},\eta,\overline{\zeta}]^td\nu_0(\eta),\  z,\zeta \in \bH
$$
where, for $z,\eta, \zeta \in \bH$, the expression $[ z,\overline{\eta},\eta,\overline{\zeta}]=
\frac{(z-\overline{\zeta})(\eta-\overline{\eta})}{(z-\overline{\eta})(\eta-\overline{\zeta})}$
is the four point function.

 In [Ra1], (see also the review in the introduction of [Ra2]) we constructed a dense unital *-algebra domain
$\D_{c_t}\subseteq \A_t$, so that right-differentiating in the $t$ parameter, the formula for the product of two symbols $k,l\in \D_{c_t}$ yields an element $c_t(k,l)$ in $\A_t$, whose reproducing kernel on $\bH\times \bH$ is given by the formula:
\begin{gather*}
c_t (k,l)(z,\zeta)=\frac{d}{ds}\left[(k\star_s l)(z,\zeta)\right]|_{s\rightarrow t_+}=\\ =
\frac{\chi'_t}{\chi_t}(k\star_tl)(z,\zeta)+\chi_t\int_{\bH}k(z,\overline{\eta})l(\eta,\overline{\zeta})
[z,\overline{\eta},\eta,\overline{\zeta}]^t \ln [z,\overline{\eta},\eta, \overline\zeta]d\nu_0(\eta), \ z,\zeta \in \bH.
\end{gather*}
The space $L^2(\A_t)$ is then a space of analytic-antianalytic functions $\bH^2$,
with norm
$$
\|k\|^2_{L^2(\A_t)}=\chi_t\iint_{F\times\bH}(k(z,\eta))^2d_\Bbb H^t(z,\eta)d\nu^2_0(z,\eta),
$$
where $$d_\Bbb H(z,\eta)=\frac{|z-\overline{\eta}|}{(\Im  z)^{1/2}(\Im  \eta)^{1/2}},$$
for $z,\eta\in\bH$, is a function depending only on the hyperbolic distance between $z,\eta\in\bH$.

The ``imaginary'' part $c_t^0$ of the deformation Hochschild cocycle (see Definition \ref{im}) is then implemented by
$Z_t=\frac12\lambda_t+i\T^{t,\Gamma}_{\Im(\ln\varphi)}$ where
$Y_t=\T^{t,\Gamma}_{\Im(\ln\varphi)}$ is the unbounded Toeplitz operator on $L^2(\A_t)$ (see Chapter 2, Definition \ref{toeplitz})
with symbol $\Im(\ln\varphi)$, and $\varphi(z,\overline{\zeta})=\Delta(z)\overline{\Delta(\zeta)}
(z-\overline{\zeta})^{1/2}$, $z,\zeta\in\bH$ and $\lambda_t=\frac{\chi'_t}{\chi_t}>0$ (here we use a principal value for the logarithm).

 As a corollary, we prove in Theorem \ref{ind}
that the operator $\T^{t,\Gamma}_{i\Im(\ln\varphi)}$ is antisymmetric with deficiency indices $(0,1)$.
This is in fact another way to express the fact that the Euler cocycle is non-trivial (see [Br])
in the bounded cohomology group $H^2_{bd}(\Gamma,\Z)$. This statement remains valid for modular subgroups
of $\PSL_2(\Z)$.

Finally we consider the space of operators on $L^2(\A_t)$  with Toeplitz symbols,
identifying $L^2(\A_t)$ with a Hilbert space of analytic functions.
Assume that  $\delta$ is a derivation, densely defined on $\A_t$,
that admits  a measurable function as a Toeplitz symbol, as in Definition \ref{toeplitz}, 
(and one mild condition that it extends with zero on a dense domain $\D$ in 
$B(H_t)$, so that $\D\cap L^2(\A_t)\subseteq\D(Y_t)$ is a core for $Y_t$, the coboundary operator defined in the previous paragraph)).

 We prove in Theorem \ref{obstruction}
that the obstruction for the cocycle $c^0_t$ of having a coboundary with nontrivial real part, cannot be removed by perturbing the coboundary operator $Z_t$ with the derivation $\delta$.

To prove all these results we consider an abstract framework,  assuming that we are  given a family $(\A_t)_{t>1}$ of   finite von Neumann algebras. Then, the symbol map is described by a Chapman Kolmogorov system of linear maps  $(\Psi_{s,t})_{s \geq t>1} :\A_t\rightarrow \A_s$ (see also [Ra2]).

We assume that the 
 Chapman-Kolmogorov system $(\Psi_{s,t})_{s\geq t>1}$
consists of completely positive, unit preserving maps. Thus, we are given
a family of finite von Neumann algebras $\A_t$, and we assume that $\Psi_{s,t}$ maps $\A_t$
injectively into $\A_s$, for $s>t$, with dense range.

The cocycle $c_t$ is in this case defined on a dense domain (if it exists) as
$c_t(k,l)=\frac{d}{ds}\Psi_{s,t}^{-1}(\Psi_{s,t}(k)\Psi_{s,t}(l))|_{s\rightarrow t_+}$
for $k,l$ in a dense domain.
Obviously, $c_t(1,1)=0$, as $(\Psi_{s,t})_{s> t}$ are unital.

It is clear that if there exists a single von Neumann algebra $\A$, and
isomorphisms $\alpha_t:\A_t\rightarrow\A$, $t>1$, so that
$\tilde{\Psi_{s,t}}=(\alpha_t^{-1}{\Psi_{s,t}}\alpha_s)_{s,t}$ is still compatible
with the differential structure, then the Hochschild cocycles
$\tilde {c_t}$,  associated to
$(\tilde{\Psi_{s,t}})_{s\geq t>1}$, are (up to the domain)
equal to $\alpha_t^{-1}c_t\alpha_t$. Obviously, in this case the generator
$\tilde{\alpha_t}=\frac{d}{ds}\tilde{\Psi_{s,t}}|_{s\rightarrow t_+}$
would contain the identity element  $1$ in its domain.

Hence the problem  described above, about the structure of the (unbounded) Hochschild cocycles $(c_t)_{t>1}$, is related
to the obstruction to deform the system $(\Psi_{s,t})_{s>t}$ into a 
Chapman-Kolmogorov system of completely positive maps, ``living'' on a single algebra.

Based on this generalization of  the obstruction in realizing the family of the symbol maps on  a single algebra,
we introduce a new invariant for finite von Neumann algebras $M$, consis\-ting of a pair $(c,Z)$,
where $c$ is an unbounded 2-Hochschild cohomology cocycle with domain $\D$, a unital $*$-algebra, and $Z$ is
an unbounded coboundary operator for $c$, with domain $\D_0\subseteq \D$ not containing the space of multiples of the identity operator $1$.

We assume that $c(1,1)=0$ and that $Z=\alpha +X+iY$ has antisymmetric, unbounded part $Y$, so that
$Y$ has deficiency indices (0,1) and moreover
$(iY^*)1=(-\alpha)1$, $\alpha>0$, and $X$ is selfadjoint with $X1=0$.
Therefore  $Z^*1=0$, and hence $1$ is forbidden to belong to the domain of $Y$ since $(iY^*)1=(-\alpha)1$.
We also assume reality for the 2-cocycle $c_t$, that is we assume that
$c(k^*,l^*)=c(l,k)^*$,  $X(k^*)=X(k)^*$, $Y(k^*)=Y(k)^*$, for $l,k\in\D$.

The new invariant for the algebra $M$ and the data $(c,Z)$   consists of the existence (or not)   of an unbounded derivation $\delta$ on $M$,
$\delta(k^*)=k^*$ for $k\in\Dom\delta$, with non-trivial real part $\Re \delta=\alpha 1$, so that $Z+\delta$
is defined at $1$ with $(Z+\delta)1=0$, and $iY+(\delta-\Re\delta)$ is selfadjoint
(here we require that $\Dom(Z)\cap \Dom(\delta)$
is a core for $Z+\delta$). 
 Here one therefore  asks   the question of whether one can (or can not) perturb $Z$ by a derivation so that the imaginary part of the cocycle $c$ admits a coboundary which is also an antisymmetric operator, that is also defined at 1.

In our example of the $\Gamma-$ equivariant quantization, for every $t>1$ there exists a $1$-parameter semigroup $(\Phi^t_\varepsilon)_{\epsilon\geq 0}$ of densely defined and completely positive
maps, on their domain,  with the property that   the coboundary operator $Z_t$ is the infinitesimal generator $\frac{d}{d\varepsilon}\Phi^t_\varepsilon|_{\varepsilon=0}$,
and the domain of $\Phi^t_\varepsilon$ is controlled by another commuting family $D_\varepsilon$
of bounded completely positive maps on $\A_t$.
Moreover, $(\overline{{\rm range}\, \Phi^t_\varepsilon(1)})_{\varepsilon>0}$
is an increasing family of projections,  increasing to 1,
as $\varepsilon\rightarrow 0$.

The semigroup $\Phi_\varepsilon^t$ has the remarkable property that
$$\Phi_\varepsilon^t(1)^{-1/2}\Phi_\varepsilon^{t}\Phi_\varepsilon^t(1)^{-1/2}$$
is an isomorphism onto its range, scaling the trace and increasing the support as $\varepsilon$ tends to 0. This is a feature that is rather common for finite von Neumann
algebras having non-trivial (full) fundamental group.

This is interesting because (see Proposition \ref{infinite}) the finite von Neumann algebra $L(F_\infty)$ associated to the free group $F_\infty$ with infinitely many generators, admits such a derivation, that is it is acting
as the derivative of a dilation on a specific subalgebra.

\section{The quantum dynamical system associated to the symbol map, 
and its associated differential structure}

In this chapter we formalize the properties of the Berezin's symbol
map, and derive some consequences.

We assume we are given a family $(\A_t)_{t>1}$ of finite von Neumann algebras 
with faithful traces $\tau_{\A_t}$ ($\tau$ simply when no confusion in possible).
The index set, the interval $(1,\infty)$ is taken here as a reference,
it could be replaced by any other left open interval.

To introduce a ``differentiable" type structure on the family $(\A_t)_{t>1}$, 
we consider a system $(\Psi_{s,t})_{s\geq t>1}$ of unital, completely positive,
 trace preserving maps $\Psi_{s,t}: \A_t\rightarrow \A_s$.

This family of completely positive maps is, in fact, an abstract framework for the Berezin symbol.
We assume that the maps $(\Psi_{s,t})_{s\geq t>1}$ are injective, with weakly dense ranges.
 By  $\Psi_{s,t}^{-1}$ we denote the unbounded inverse of the continuous extension of $\Psi_{s,t}$ to $L^2(\A_t)$. We assume that all the operations of  differentiation in the 
$t$-parameter, involving the maps $\Psi_{s,t}^{-1}$,  have dense domains. This will be the case in our main 
example coming from the Berezin $\Gamma$-equivariant
quantization. Being completely positive maps, we have that $\Psi_{s,t}(k^{\star})=\Psi_{s,t}(k)^{\star}$,
for all $k$ in $\A_s$, $s\geq t>1$.

In addition, we assume that the system $(\Psi_{s,t})_{s\geq t>1}$ verifies 
the Chapman-Kolmogorov relations $\Psi_{r,s}\Psi_{s,t}=\Psi_{r,t}$, if $r\geq s\geq t>1$.

The algebras $(\A_t)_{t>1}$ are also assumed to be Morita equivalent, in a functorial
way, with the growth of the comparison maps (for the Morita equivalence)
 controlled by the absolute value $|\Psi_{s,t}|$ of the completely positive maps $(\Psi_{s,t})_{s\geq t>1}$, viewed as contractions acting on $L^2(\A_t)$.

We introduce a set of assumptions. The first assumption describes the  Morita equivalence between the algebras $(\A_s)_{s>1}$.
In the sequel we will seldom  identify a completely positive map on $\A_t$ with its extension to $L^2(\A_t)$. 
 
\vskip6pt

{\sc Assumption F.M.}
We are given an increasing, differentiable function $\chi_t$, $t>1$, with
strictly positive values. In the particular of the Berezin's quantization (see  [Ra1]) the function  $\chi_t$, depending on $t$,
is a linear  function depending on the Plancherel coefficient of the projective unitary representation $\pi_t$.

We assume that we have a family of injective completely positive maps 
$$\beta_{s,t}:\A_t\rightarrow E^s_t \A_s E^s_t, \text{\rm \ for \ } s\geq t>1,$$
 where, for fixed $s$, $(E_t^s)_{1<t\leq s}$ is a family 
of increasing (selfadjoint) projections in $\A_s$ 
of size (trace) $\frac{\chi_t}{\chi_s}$.

 Observe that, for $1<t\leq s$,  $\beta_{s,t}(1)\in \A_s$ is a positive element,
 which we denote by
$X^s_t\in(\A_s)_+$, $0\leq X^s_t\leq1$, $X^s_t\in E^s_t\A_s E^s_t$. We assume that 
$X^s_t$ has zero kernel for all $s\geq t$. By $(X^s_t)^{-1}$ we denote the (eventually unbounded) inverse of the operator $X^s_t$ on its support. 

We define, for $1<t\leq s$, $a\in \A_t$, $$\alpha_{s,t}(a)=(X^s_t)^{-1/2}\beta_{s,t}(a)(X^s_t)^{-1/2}.$$

We are also assuming that $\alpha_{s,t}$   is a surjective isomorphism from $\A_t$ onto $E^s_t\A_sE^s_t$. In the specific case of the Berezin quantization, the operator $X^s_t$ is the absolute value of an intertwiner for the corresponding Hilbert spaces, so that $\alpha_{s,t}$ is automatically  an isomorphism onto its image (see Chapter 6 in [Ra2]).

Thus $\beta_{s,t}(a)=(X^s_t)^{1/2}\alpha_{s,t}(a)(X^s_t)^{1/2}, a\in \A_t$.
If $Y^s_t\in \A_t$ is the pre-image of $X^s_t$ through the morphism $\alpha_{s,t}$,
$s\geq t$, then 
$$
\beta_{s,t}(a)=\alpha_{s,t}((Y^s_t)^{1/2}a(Y^s_t)^{1/2}),\quad 
a\in \A_t. 
$$
We assume that the family of completely positive maps $(\beta_{s,t})_{1<t\leq s}$ verifies the Chapman-Kolmogorov condition. Thus, we are assuming that   $$\beta_{r,s}=\beta_{r,t}\circ\beta_{t,s}, {\rm \ for\ }
r\geq t\geq s>1.$$

\

\

In addition, we are  assuming that the two families of  completely positive maps linear maps are correlated,
by a semigroup of ``quasi positive" maps, as explained in the next assumption. By "quasi positive" map we mean a densely defined map which is completely positive on its domain.
More precisely:

\vskip6pt

{\sc Assumption SQP.} We assume that the  diagram  
$$
\begin{array}{cccccc}
\A_{t-\varepsilon} & 
\mathop{\longrightarrow}\limits^{\Psi_{t,t-\varepsilon}} & \A_t & \mathop{\longrightarrow}\limits^{\Psi_{t+\varepsilon,t}} & \A_{t+\varepsilon} \\
\hspace{-0,8cm}{}_{\Phi_{\varepsilon}^{t-\varepsilon}} \Big\uparrow & 
\nearrow {}_{\beta_{t,t-\varepsilon}}  & 
\Big\uparrow{}_{\Phi_{\varepsilon}^t} & \nearrow {}_{\beta_{t+\varepsilon,t}} &  \Big\uparrow {}_{\Phi_{\varepsilon}^{t+\varepsilon}} \\ [8pt]
\A_{t-\varepsilon} & 
\mathop{\longrightarrow}\limits^{\Psi_{t,t-\varepsilon}} & \A_t & \mathop{\longrightarrow}\limits^{\Psi_{t+\varepsilon,t}} & \A_{t+\varepsilon} \\ 
\end{array}
$$
is commutative, where $\Phi^t_\varepsilon$ is densely defined (the domain will be made explicit in the Assumption SQP1), and it is  uniquely defined by the formula
$$
\Phi^t_\varepsilon=\Psi^{-1}_{t+\varepsilon,t}\circ\beta_{t+\varepsilon,t}=\beta_{t,t-\varepsilon}\circ\Psi_{t,t-\varepsilon}^{-1},\quad \varepsilon>0.
$$ 
Thus we are also requiring the above equality. 

We implicitly assume that $\Phi_\varepsilon^{t+\varepsilon}$ 
defined as $\beta_{t+\varepsilon,t}\circ\Psi_{t+\varepsilon,t}^{-1}$ and similarly for
$\Phi^{t-\varepsilon}_{\varepsilon}\!$, $\Phi^{t}_{\varepsilon}\!$,
are making the above  diagram commutative.

\

\
The next two assumptions describe  the growth  of the 
absolute value of the completely positive maps in the Chapman-Kolmogorov system.

\vskip6pt 

{\sc Assumption D.}   For $s\geq t>1$, the linear positive maps 
$D^t_{s-t}=\Psi_{s,t}^*(\Psi_{s,t})$ on $L^2(\A_t)$, where the completely positive maps $\Psi_{s,t}$ are uniquely extended by continuity to $L^2(\A_t)$, 
have the property that $(D^t_{\varepsilon})_{\varepsilon\geq0}$
is a commuting family of positive operators acting on  $L^2(\A_t)$.
In the case of the Berezin's quantization this is simply the Toeplitz operator  $\mathcal T^{t,\Gamma}_{d_\Bbb H^{\varepsilon}}$ on $L^2(\A_t)$ with symbol $d_\Bbb H^{\varepsilon}$, with $d_\Bbb H$ as in the introduction section (see Definition \ref{toeplitz} for the definition of the Toeplitz operator).

\vskip6pt

\vskip6pt 

{\sc Assumption D$'$.}   For $r\leq t$, there exists a  family of unbounded, densely defined, closed, injective,   positive maps 
$D^t_{r-t}$ acting  on a dense subspace of $L^2(\A_t)$, that have the property that the domain of $D^t_{r-t}$ is the image of $\Psi_{tr}$. In addition  $((D^t_{-\varepsilon})^{-1})_{\varepsilon\geq0}$
is a commuting family, of bounded operators commuting with $(D^t_\varepsilon)_{\varepsilon\geq0}$.
Thus the image of $D^t_{-\varepsilon}$ coincides with the image of 
$\Psi_{t-\varepsilon, t}$ inside $L^2(\A_t)$. 

In the case of the Berezin's quantization, the operator $D^t_{-\varepsilon}, \varepsilon >0$  is simply the (unbounded) Toeplitz operator $\mathcal T^{t,\Gamma}_{d_\Bbb H^{-\varepsilon}}$ on $L^2(\A_t)$ with symbol $d_\Bbb H^{-\varepsilon}$. 

\

\

The Toeplitz operators considered in Assumptions D, D$'$, in the case of the Berezin quantization,  are simply functions of the invariant Laplacian, when using the Berezin transform (see [Ra1])).

\vskip6pt
We  complete the  Assumption SQP by the Assumption SQP$_1$, which describes the domain of the semigroup $(\Phi^t_\varepsilon)_{\varepsilon \geq 0}$ for fixed $t>1$.

\vskip6pt
{\sc Assumption SQP$_1$}. The composition $\Phi^t_\varepsilon\circ D^t_{-\varepsilon}$
is continuous, and maps positive into positive elements. This assumption  will be proven later in this chapter to hold true in the case of the Berezin quantization.

In particular, the domain of $\Phi^t_\varepsilon$ is the image of $D^t_{-\varepsilon}$.

\

If all  Assumptions FM,
D, D$'$, SQP, SQP$_1$ are verified, we introduce the following definition.

\begin{defn}
The symbol system $(\A_t)_{t>1}$,  $(\Psi_{s,t})_{s\geq t>1}$, $(\beta_{s,t})_{s\geq t>1}$
with all  assumptions listed 
above (FM,
D, D$'$, SQP, SQP$_1$) is regular, if the following operations obtained by differentiation,
have a dense domain $\D=\D_{c_t}$.
We define the Hochschild cocycle associated to the deformation by the formula
$$
c_t(k,l)=\frac{d}{ds}\Psi_{s,t}^{-1}(\Psi_{s,t}(k)\Psi_{s,t}(l))|_{s\rightarrow t_+},\  k,l \in \D.
$$
(by $s\rightarrow t_+$  we denote the right  derivative at $t$ of the above expression).

The Dirichlet form associated to the deformation is
$$
\E_t(k,l)=\frac{d}{ds}\tau_{\A_s}(\Psi_{s,t}(k)\Psi_{s,t}(l))|_{s\rightarrow t_+}, \ \  k,l \in \D,
$$
and the real part of the
cocycle is 
$$
\langle R_tk,l\rangle_{L^2(\A_t)}=
-\frac12\E_t(k,l^*)=
\frac{d}{ds}\langle \Psi_{s,t}(k),\Psi_{s,t}(l)\rangle_{L^2(\A_s,\tau_{\A_s})},\ \  k,l \in \D.
$$

Note that, as proved in [Ra2], the bilinear form $\E_t$ is indeed a Dirichlet form in the sense considered in the papers
[Sau], [CS].

\begin{defn}\label{im}
The imaginary part of the cocycle associated to the deformation is defined by the 
formula
$$
c_t^0(k,l)=c_t(k,l)-[R_t(kl)-kR_t(l)-R_t(k)l], \ \ k,l \in \D.
$$
All these equations are assumed  satisfied on the dense domain $\D$. 

\

Note that the definition for $c_t^0(k,l)$ is so that $c_t^0$  remains a Hochschild cocycle, 
and in addition $\tau(c_t^0(k,l))=0$ for all $k,l$ in the domain. Moreover, as proved in [Ra1],
[Ra2], the trilinear form defined by
$$
\psi_t(k,l,m)=\tau_{\A_t}(c_t^0(k,l)m)
$$
is densely defined and has the property that 
$$
\psi_t(k,l,m)=\psi_t(m,l,k)=\overline{\psi_t(l^*,k^*,m^*)},
$$
and is a $2$-cyclic cohomology cocycle ([Co]).

Both $c_t$, $c_t^0$ have the property $c(k^*, l^*)=c(l,k)^*$ and similarly for $c_t^0$, as also $R_t(k^*)=R_t(k)^*$.
\end{defn}

The analysis between the two types of completely positive maps, $\Psi_{s,t}$ and $\beta_{s,t}$, allows us to construct  an unbounded coboundary for 
the cocycle $c_t$. The domain of this coboundary, which we denote by $\D^0$ is slightly smaller than the previous domain $\D$. In the case of the Berezin's quantization ([Ra2]) we constructed explicitly $\D_0$, as a $\ast$- algebra that  does not contain the identity. The domain $\D_0$ will be the common domain for
the derivation operations.

 However, by definition, the cocycles $c_t$, $c_t^0$, $\psi_t$ naturally contain the identity element  
in their natural domain, and  $c_t(1,1)=0$, $c_t^0(1,1)=0$, $\psi_t(1,1,1)=0$.

We do the analysis of the relation between the two families of completely positive maps  in the following theorem.

\end{defn}

\begin{thm}\label{ind}
 We assume  all the above  assumptions  on the system \break $(\A_t,\Psi_{s,t},\beta_{s,t})_{s\geq t>1}$. In addition we assume that 
 the generator $\L_t$ of the semigroup $\Phi_\varepsilon^t$ at $\varepsilon=0$
has a dense   domain, which is a dense $\ast$-algebra $\D^0=\D_{\L_t}\subseteq \D$. 
Then, 
$$
c_t(k,l)=\L_t(kl)-k\L_t(l)-\L_t(k)l-kT^tl , \ \  k,l\in \D_{\L_t}.
$$
Here $T^t$ is the derivative $\frac{d}{d\varepsilon}Y_t^{t+\varepsilon}|_{\varepsilon=0}$, an unbounded operator affiliated with $\A_t$. We   assume that $T^t$ has a dense domain in $H_t$  
with the property   that $kT^tl$ is bounded for $k,l$ in $\D_{\L_t}$.
Recall that the operator $Y_t^{t+\varepsilon}$ was introduced in Assumption F. M.

Note that as a consequence, the cocycle $c_t$ is implemented by an unbounded 1-cocycle, defined on $\D_{\L_t}$,  of the form
  $$\Lambda_t= \L_t-\frac12\{T^t,\cdot\}.$$

In the case of the Berezin's quantization 
this is the Lindblad form (see [Li], [CE])
of a generator of a quantum dynamical semigroup { (see also the references in  [Ra2] )}.

Then $\L^0_t=\L_t-R_t$ has the property~that
$$
c_t^0(k,l)=\L_t^0(kl)-k\L_t^0(l)-\L_t^0(k)l, \ \  k,l\in \D_{\L_t}
$$
and  $$\L^0_t=\lambda_t1+iY_t,$$
 where $Y_t$ is symmetric, $(\L^0_t)^*1=0$ 
(and hence $(iY_t)^*1=\lambda_t1$). In addition, $Y_t$ has 
deficiency indices $(0, 1)$. 
Here $\lambda_t=-\frac{1}{2}\frac{\chi_t'}{\chi_t}$.
\end{thm}

\begin{proof}
The domain of the semigroup $\Phi_{\varepsilon}^t$ , is defined abstractly, thus it contains the range of
integrals of the form $\int^\beta_\alpha\Phi^t_\varepsilon d\varepsilon$.

Note that the family $X_r^t=\beta_{r,t}(1)$ is an increasing family of positive elements.
Indeed, if $r_1<r_2<t$, 
$$
X_{r_1}^t=\beta_{r_1,t}(1)=\beta_{r_2,t}(\beta_{r_1,r_2}(1))=
\beta_{r_2,t}(X_{r_2}^{r_1})\leq\beta_{r_2,t}(1)=X_{r_2}^t.
$$
The rest of the statement was essentially proved, for the particular case of the Berezin
quantization in [Ra2].

We prove again the statement here, in the general framework.
All derivatives are computed on the domain $\D_{\L_t}$.

Since $\beta_{s,t}(a)=\alpha_{s,t}\left((Y^t_s)^{1/2}a(Y^t_s)^{1/2}\right)$, $a\in \A_t$
where $Y^t_s\in \A_t$ is the positive element previously defined, it follows that 
for $s\geq t$, the following equality holds: 
$$
\beta_{s,t}(k)\beta_{s,t}(l)=
\alpha_{s,t}\left((Y^t_s)^{1/2}kY_s^tl(Y^t_s)^{1/2}\right)=
\beta_{s,t}(kY_s^tl)
$$ 
for $k,l$ in the maximal domain, on which the 
differentiation is possible (which a posteriori is $\D_{\L_t}$).

We apply $\Psi_{s,t}^{-1}$, and obtain that
$$
\Psi_{s,t}^{-1}(\beta_{s,t}(kY_s^tl))=
\Psi_{s,t}^{-1}([\Psi_{s,t}(\Psi_{s,t}^{-1}(\beta_{s,t}(k)))]
[\Psi_{s,t}(\Psi_{s,t}^{-1}(\beta_{s,t}(l)))]),
$$
and hence 
$$
\Phi_{s-t}^t(kY_s^tl)=
\Psi_{s,t}^{-1}([\Psi_{s,t}\circ\Phi_{s-t}^{t}(k)]
[\Psi_{s,t}\circ\Phi_{s-t}^{t}(l)]).
$$
Differentiating in the parameter $s$ when $s\rightarrow t_+$, we obtain 
$$
\L_t(kl)+kT^tl=c_t(k,l)+k\L_t(l)+\L_t(k)l.
$$
Thus for $k,l$ in the domain of $\L_t$, we have that
$$
c_t(k,l)=\L_t(kl)-k\L_t(l)-\L_t(k)l+kT^tl.
$$
We analyze the real and imaginary part of $\L_t$. By definition $\L_t$ is the 
generator of the semigroup 
$\Phi^t_\varepsilon=\Psi^{-1}_{t+\varepsilon,t}\circ\beta_{t+\varepsilon,t}$.
Thus
\begin{gather*}
\langle \Re\L_t k,l\rangle_{L^2(\A_t)}=
\frac12\frac{d}{ds}\left[\langle \Psi_{s,t}^{-1}(\beta_{ts}(k)),\Psi_{s,t}^{-1}(\beta_{ts}(l))\rangle_{L^2(\A_t)}\right]|_{s\rightarrow t_+}=\\
=\frac12\left[-\frac{d}{ds}\langle \Psi_{s,t}(k),\Psi_{s,t}(l)\rangle_{L^2(\A_s)}+
\frac{d}{ds}\langle \beta_{s,t}(k),\beta_{s,t}(l)\rangle_{L^2(\A_s)}\right]|_{s\rightarrow t_+}.
\end{gather*}

Using further that $\beta_{s,t}(k)=\alpha_{s,t}\left((Y^t_s)^{1/2}k(Y_s^t)^{1/2}\right)$
 we obtain that this is further equal to 
$$
\frac12\left[-\E_t(k,l^*)+\langle T^tk+kT^t,l \rangle -\frac{\chi'_t}{\chi_t}
\langle k,l \rangle_{L^2(\A_t)}\right].
$$

The last term is due to the fact that
$\langle \alpha_{s,t}(k), \alpha_{s,t}(l)\rangle$ is $\frac{\chi_t}{\chi_s}\langle k,l\rangle $, because of the trace scaling property ($\alpha_{s,t}$ is non-unital).

Moreover, the remaining part is the antisymmetric part, which is $\Im \L_t$.

We have thus obtained that 
$$
\Re\L_t=R_t+\frac12\{T^t,\cdot\}-\frac12\frac{\chi'_t}{\chi_t}{\rm Id}
$$
and hence $c_t(k,l)$ is implemented by
 $\L_t-\frac12\{T^t,\cdot\}=$
 $R_t-\frac12\frac{\chi'_t}{\chi_t}+i\Im \L_t$.

Consequently, the unbounded linear map $\Lambda_t^0=-\frac12\frac{\chi'_t}{\chi_t}+i\Im \L_t$ 
implements $c_t^0$. For simplicity, we denote
$\L_t^0=i\Im \L_t$, which is an antisymmetric~form. 

Moreover, we note that $\L^0_t$ is in fact obtained from the derivative
of $\Psi^{-1}_{s,t}(\alpha_{s,t})$ after taking into account the rescaling due to the
variation of $\langle\cdot\,,\cdot\rangle_{L^2(\A_s)}$. Since 
$\left(\frac{\chi_t}{\chi_s}\right)^{1/2} \alpha_{s,t}(k)$ 
is an isometry form $L^2(\A_t)$ onto $E^s_tL^2(\A_s)E^s_t$, it follows
that $\L^0_t$ is antisymmetric, $(\L^0_t)^*1=-\frac{1}{2} \frac{\chi'_t}{\chi_t}$ 
and the deficiency indices are $(0, 1)$ (since $\alpha_{s,t}$ is surjective 
onto $E^s_t)$. In particular, $$(\Lambda^0_t)^*=-\frac{1}{2}\frac{\chi'_t}
{\chi_t}-i \Im  \L_t^*=-\frac{1}{2}\frac{\chi'_t}{\chi_t}+(\L^0_t)^*,$$ and hence 
$(\Lambda^0_t)^*1=0$. 
\end{proof}

\

We explain all the above results and assumptions in the case 
of the Berezin $\Gamma$-equivariant quantization of the upper halfplane.

In this case we have $\Gamma=\PSL_2(\Z)$ and we let $\pi_t:\PSL_2(\R)\rightarrow B(H_t)$ 
be the discrete series (extended, for 
$t>1$, to  the continuous family of projective unitary representations of $\PSL_2(\R)$ as in  [Pu]).

By taking $\A_t=\{\pi_t(\Gamma)\}'\subseteq B(H_t)$, it is proved in [Ra], [GHJ] that $\A_t$ is isomorphic to the  von Neumann algebra associated to the group $\Gamma$ with cocycle $\varepsilon_t$, reduced by $\frac{t-1}{12}$ (see [GHJ]), that is the algebra
$L(\Gamma,\varepsilon_t)_\frac{t-1}{12}$ (here $\varepsilon_t$ is the cocycle
coming from the projective representation $\pi_t$).

Every kernel (function) $k:\bH\times\bH \rightarrow \bC$, analytic in the first variable
antianalytic in the second variable, corresponds, if certain growth conditions are verified
(see [Ra1], Chapter 1) to a bounded operator on $B(H_t)$, which has exactly the reproducing kernel $k$. If 
$k(\gamma z, \overline{\gamma\zeta})=k(z,\zeta)$, 
$\gamma\in\Gamma$, $z,\zeta\in \bH$, then
the corresponding operator is in $\A_t$.

The following data is then computed directly in the space of kernels. We
identify kernels $k, l$ with the corresponding operators and denote y b$k\star_tl$ the kernel  
(symbol) of the product.

The formulae for the various operations in $\A_t$, are computed as follows:

For $k, l$ kernels (symbols) of operators in $\A_t$, we have 
$$\tau_{\A_t}(k)=\frac{1}
{\rm ha(F)}\int_F k(z,\overline{z})d\nu_0(z).$$ Here $F$ is a fundamental
domain for the action of $\Gamma$ in $\bH$, and $\rm{ha}(F)$ is the hyperbolic
area of F. Moreover, the product formula for the kernel of the product operator, of the operators represented by the kernels $k,l$  is:
\begin{gather*}
(k\star_tl)(z,\overline \zeta)=\chi_t(z-\overline{\zeta})^t\int_{\bH}\frac{k(z,\overline{\eta})l(\eta,\overline{\zeta})}{(z-\overline{\eta})^t(\eta\overline{\zeta})^t}d\nu_t(\eta)=\\
=\chi_t\int_{\bH}k(z,\overline{\eta})l(\eta,\overline{\zeta})[z,\overline{\eta},\eta,\overline{\zeta}]^t
d\nu_0(\eta),
\end{gather*} 
where 
$[z,\overline{\eta},\eta,\overline{\zeta}]=\frac{(z-\overline{\zeta})(\eta-\overline{\zeta})}
{(z-\overline{\eta})(\eta-\overline{\zeta})}$, 
$z,\zeta,\eta\in \bH$ is the four point function.

The Hochschild cocycle, obtained by derivation of the product formula, for $k,l$ in the domain $\D_{c_t}$ is given by 
$$
c_t(k, l)(z,\zeta)= \frac{\chi'_t}{\chi_t}+ c_t\int_\bH k(z,\overline{\eta})
l(\eta,\overline{\zeta})[z,\overline{\eta},\eta,\overline{\zeta}]^t \ln([z,\overline{\eta},\eta,
\overline{\zeta}])d\nu_0(\eta).
$$
Here we use a standard branch of the logarithm for $\ln(z-\overline{\zeta})$,
$z, \zeta\in \bH$.

Consequently,
$$
\tau_{\A_t}(c_t(k,l))=\frac{d}{ds}\left[\tau_{\A_s}(k\star_s l)\right]|_{s\rightarrow t_+}=$$

$$\frac{\chi'_t}{\chi_t}\langle k, l^\ast\rangle_{L^2(\A_t)}+\int_F\int_{\bH} 
k(z,\overline{\eta})l(\eta,\overline{z})d_{\Bbb H}(z,\eta)^t \ln d_{\Bbb H}(z,\eta)\, d\nu_0^2(z,\eta).
$$ 

Moreover:
$$
\langle k,l \rangle_{L^2(\A_t)}=
\chi_t\iint_{F\times\bH} k(z,\overline{\eta})\overline{l(\eta,\overline{z})}d_\Bbb H^t(z,\eta)d\nu^2_0(z,\eta),
$$
so
$$
\|k\|^2_{L^2(\A_t)}=
\chi_t\iint_{F\times\bH} |k(z,\eta)|^2 d_\bH^t(z,\eta)d\nu_0^2(z,\eta)
.$$
In particular $L^2(\A_t)$ is a space of functions with analytic 
structure, analytic in the first variable, antianalytic in the 
second, and diagonally $\Gamma$-invariant, square summable, with
respect to the measure $d_\Bbb H^t(z,\eta)d\nu^2_0(z,\eta)$.

To construct the rest of the data, we let $\Psi_{s,t}:\A_t\rightarrow \A_s$ be simply the symbol map,
associating to a bounded operator with kernel $k$ in $\A_t$, the bounded operator
with the same kerned $k$ in $\A_s$ (see [Ra1]).

The family of completely positive maps $\beta_{s,t}$ are constructed from intertwiners between
various representations $\pi_t$, $\pi_{t+n}$, $t>1$, $n\in\N$, obtained from automorphic cusp forms.

Namely, if $f, g$ are  automorphic cusp forms of order $n$ for the group $\PSL_2(\Z)$, then let
$M_f^t:H_t\rightarrow H_{t+n}$ be the operator of multiplication by $f$. Because of the results in ([GHJ]) this is a bounded intertwiner
between the two representations $\pi_t$, $\pi_{t+n}$,  and hence we obtain a map
$$\beta_{f,g}:\A_t\rightarrow \A_{t+n},$$
defined by the formula
$$
\beta_{\overline{f},g}(k)=M_g^tk(M_f^t)^\ast.
$$
This is a bounded operator from $\A_t$ into $\A_{t+n}$. We recall the following fact.

\begin{prop}[{[Ra1], [Ra2]}] Given $f, g$ as above, of order $n$, and $t>1$,
the symbol of $\beta_{\overline{f},g}(k)$ is 
$$
\frac{\chi_t}{\chi_{t+n}}f(z)\overline{g(\zeta)}k(z,\overline{\zeta}),\quad z,\zeta\in\bH.
$$
\end{prop}

In particular, if $f=g$, then $\beta_{\overline{f},f}$ is a completely positive map, and if
$x_fv_f$ is the polar decomposition of $M^t_f$, then $\beta_f(k)=x_f(v_fkv_f^*)x_f$
and $\alpha_f(k)=v_fkv_f^*$ defines an isomorphism from $\A_t$ into $E_f\A_{t+n}E_f$, 
where $E_f\in\A_{t+n}$ is the range of $v_f$.

We do this construction for $\Delta^\varepsilon$, $\varepsilon>0$, where $\Delta$ is the
Dedekind function (a cusp form of order 12). Since $\Delta$ has no zeros,
$\ln\Delta,\Delta^\varepsilon$ are uniquely defined, and we define
$$
\beta_\varepsilon=\beta_{\overline{\Delta^\varepsilon},\Delta^\varepsilon}
:\A_t\rightarrow E_t^{t+\varepsilon}\A_{t+\varepsilon}E_t^{t+\varepsilon},
$$
where $E_t^{t+\varepsilon}$ is the closure of the range of $M_{\Delta^\varepsilon}$.
Since $M^t_{\Delta^\varepsilon}$ is injective, as $\Delta^\varepsilon$ has no 
zeros, it follows that $\tau(E^{t+\varepsilon}_t)=\frac{\chi_t}{\chi_{t+\varepsilon}}$.

We note the following interesting corollary.

\begin{cor}
The algebras $L(\Gamma,\varepsilon_t)$, $t>1$, are Morita equivalent.
\end{cor}

\begin{proof} By the above we have that $(\A_{t+\varepsilon})_{\frac{\chi_t}{\chi_{t+\varepsilon}}}=\A_t$ since $c_s=\frac{s-1}{\pi}$ and $\A_s=L(\Gamma,\varepsilon_s)_{\frac{s-1}{12}}$, for $s>1$, so the 
result follows.
\end{proof}

\

We return to the description of the $\Gamma$-equivariant Berezin's quantization and the description of the constitutive elements, as used in the assumptions made for Theorem 3. It follows that the formula for $\Phi^t_\varepsilon$ is 
$$
\Phi_\varepsilon^t(k)=\frac{\chi_t}{\chi_{t+\varepsilon}}(k\varphi^\varepsilon),\quad k\in\D_{\Phi_t},
$$
where $$\varphi(z,\overline{\zeta})=\ln((z-\overline{\zeta})^{12})+\ln\Delta(z)+
\ln(\overline{\Delta(\zeta)}), z,\zeta\in\bH,$$
 and where we use the standard analytic branch of the logarithm of the non-zero, analytic function $\Delta$. The domain of the maps $\Phi^t_\varepsilon$, $\varepsilon \geq 0$ was constructed explicitly in Chapter 6 in ([Ra2]), but it can also be described in a more axiomatic way, as in assumption SQP1.

The condition that $\Phi_\varepsilon^t\circ D_{-\varepsilon}^t$ is
bounded follows from  
\begin{gather*}
|\Delta(z)\Delta(\overline{\zeta})(\overline{z}-\zeta)^{12}|^2=\\
|\Delta(z)|^2(\Im z)^{12} |\Delta(\zeta)|^2(\Im \zeta)^{12}
\left(\frac{|(\overline{z}-\zeta)|^2}{(\Im z)\Im\zeta}\right)^{12},
\end{gather*}
which is a quantity bounded by $d_\Bbb H^{12}$.

Finally, we note (see also Chapter 6 in [Ra2]) that the domain  of $$\L_t=
\frac{d}{d\varepsilon}\left[\Psi_{t+\varepsilon,t}^{-1}(\beta_{t+\varepsilon,t})\right]|_{\varepsilon=0}=\frac{d}{d\varepsilon}\left[\beta_{t,t-\varepsilon}(\Psi_{t,t-\varepsilon}^{-1}(k))\right]|_{\varepsilon=0},
$$ 
contains all integrals of the form
$$
\int_{\varepsilon_1}^{\varepsilon_2}M_{\Delta^\varepsilon}kM_{\Delta^\varepsilon}d\varepsilon
$$
if $\varepsilon_0<\varepsilon_1<\varepsilon_2$ and $k$ belongs to $\A_{t-\varepsilon_0}$
(rigorously, in the integral we are using $\Psi_{t-\varepsilon,t-\varepsilon_0}(k)$ instead of $k$).

The formula for $\L_t$ as described in ([Ra2] Chapter 6) corresponds to a Toeplitz operator, as in Definition \ref {toeplitz}, of multiplication with an analytic
function. 

More precisely we have that $$\L_t=-\frac{\chi'_t}{\chi_t}1+\M_{\varphi},$$ where by $\M_\varphi$ we denote the ``analytic" Toeplitz
operator of multiplication by $\varphi$, (in the terminology of Definition\ref{toeplitz} this is $\T^{t,\Gamma}_{\varphi})$.

It follows that $\L^0_t$ (the imaginary part of $\L_t$)  is a Toeplitz operator (in the sense of Definition \ref{toeplitz} on the Hilbert space $L^2(\A_t)$
 identified with a space of analytic, antianalytic
functions -respectively in the first and in second variable.

Thus $\langle \L_t^0k,l \rangle=\langle\M_{\varphi_0}k,l \rangle_{L^2(\A_t)}$ where 
$\M_{\varphi_0}$ is the multiplication operator with the function $\varphi_0$, acting on
$L^2(\A_t)$. Consequently $ \L_t^0$  is the Toeplitz operator with symbol $$\varphi_0(z,\overline{\zeta})=i[\arg(z-\overline{\zeta})^{12}
+\arg(\Delta)(z)-\arg(\Delta)(\zeta)],$$

\begin{cor}Recall that,
for $t>1$, the Hilbert space  $L^2(\A_t)$  is identified with  the Hilbert space of analytic-antianalytic functions $k: \bH\times\bH \rightarrow \bC$ (in the first, 
respectively the second variable) diagonally $\Gamma$-invariant, with
norm
$$
\|k\|_{L^2(\A_t)}=
\chi_t\iint_{F\times\bH}|k(z,\eta)|^2d_\bH^t(z,\eta)\, d \nu^2_0(z,\eta).
$$
Then the Toeplitz operator $T_0=\T^{t,\Gamma}_{\varphi_0}=iY^t$ (see Definition \ref{toeplitz}), having a  dense domain as constructed in Corollary 6.6 in [Ra2]), with symbol $$\varphi_0=i(\arg(z-\overline{\zeta})^{12}+\arg\Delta(z)-\arg\Delta(\zeta)),$$
is antisymmetric, with deficiency indices $(0, 1)$, and 
$$T_0^*1=-\frac{\chi'_t}{\chi_t}1=-\frac{1}{t-1}1.$$

\end{cor}

We note  the following problem that arises naturally, that we  formulate below, in abstract setting:  determine when 
can one perturb the operator $\Lambda^0_t=-\frac{1}{2}\frac{\chi'_t}{\chi_t}1+iY_t$ by a densely defined
derivation, with the same domain, so that the real part (which is a multiple of the identity) in the coboundary operator $\Lambda^0_t$,  
for $c^0_t$,
is removed.

\begin{prob}
Let $M$ be a type II$_1$ factor with trace $\tau$, let $Y$ be a symmetric operator
with dense domain a $\ast-$algebra $\D^0=\D^0_Y$ in $L^2(M,\tau)$
such that $(iY)^*1=-\alpha 1$,  where $\alpha >0$, and so that $Y$ has deficiency indices $(0, 1)$. Here the adjoint is taken with respect to the action on the  Hilbert space $L^2(M,\tau)$. We also assume that $Y$ takes values in $M$ rather than $L^2(M,\tau)$.

Let $Z=\alpha1+iY$, so that  $Z^* 1=0$.
Let $$c^0(k, l)=Z(kl)-kZ(l)-Z(k)l, \ \  k, l \in \D^0_Y.$$ Assume in addition that the identity element $1$ may be adjoined to the domain of $c^0$ by taking graph closure and that $c^0(1, 1)=0$.
This is an additional hypothesis, as $Y$ is not defined at $1$.
Note that here we are automatically assuming that $c^0$ coincides with its imaginary part in the sense of Definition \ref {im}. 

The problem is then stated as follows:

Determine if there exists a  densely defined derivation $\delta$ with dense $\ast-$ algebra domain $\D(\delta)$, so that 

i)$Y_1= Y+\delta$ is densely defined and $\D^0_Y\cap\D(\delta)$ is a core for $Y_1$.

ii) The real part of $\delta$, in the sense of Hilbert space operators acting on $L^2(M,\tau)$, is  $\Re\delta=-\alpha 1$ (note  that in this case $\alpha 1 +iY_1+\delta$  is antisymmetric).

iii) $(iY_1)^{\ast}(1)=0$), and  $1$
is in the domain of the closure of $Y_1$.

Note this amounts to find a coboundary  for $c^0$ that, as an operator acting on $L^2(M,\tau)$, is antisymmetric with (0,1) deficiency indices.
\end{prob}
\

\

For the cocycle $c^0_t$ constructed explicitly in Chapter 1 for the Berezin quantization, this represents
the obstruction to replace the operator $\L^0_t$ by an antisymmetric operator, 
that is to find $\L_t$ so that $\L_t1=\L^*_t1=0$ and so that $\L_t$ implements $c_t$ and 1 in the domain
of $\L_t$.  

This would happen, for example, if there  exists a unique  finite von Neumann algebra $\A$ and there exists $\alpha_t:\A_t\rightarrow\A$, so that $\alpha_t$
are all isomorphisms. In this case we may   transfer the maps $\Psi_{s,t}$, $s\geq t>1$, using the isomorphisms $\alpha_t$
into a Chapman-Kolmogorov system of unital completely positive maps $\tilde{\Psi_{s,t}}$ acting
on $\A$.

Then  the corresponding Hochschild cocycle $\tilde{c_t}$ is simply implemented by 
$\tilde{\L_t}=\frac{d}{ds}\tilde{\Psi_{s,t}}|_{s=t}$, which obviously is $0$ at $1$ and $1$ is in the domain of $\tilde{\L_t}$. 
To see this, one  could apply again the argument in Theorem \ref{ind} , with $\tilde{\beta_{s,t}}={\rm Id}$, 
as all algebras are now equal.

The following remark  shows why  the condition $c(1, 1)=0$ is necessary.

\begin{rem}
Assume $c$ on $M$ is an unbounded cocycle with $\tau(c(k, l))=0$.
If 1 is in the domain of $c$, then $c(1, 1)=0$.
\end{rem}

\begin{proof} From the cocycle condition we have that $m\,c(1, 1)=c(m, 1)$.
Hence $\tau(m(c(1, 1))=0$ for all $m$ in the domain, and thus $c(1, 1)=0$, as
we assumed that  the domain is dense.
\end{proof}

Finally we show  that the structure and the properties of the deformation $(\A_t)_{t>1}$ and of the corresponding
maps $\beta_{s,t}$, $\Psi_{s,t}$, $s\geq t>1$,  can be deduced
from  the ``rigged" Hilbert space structure on $(H_t)_{t>1}$, determined by a specific chain of embeddings, as described below.

For simplicity, we use  the model of the unit disk, so that the Berezin quantization is realized using the Hilbert spaces:
 $$H_t=H^2(\bD,(1-|z|^2)^{t-2}d\overline{z}dz.$$
 
  We describe the properties of this ``rigged" Hilbert space structure. Note that this could be used to obtain an even more general abstract formulation of the Berezin quantization.

\begin{defn}
 Assume we have a family of unitary, projective, representations $\pi_t$ of $\Gamma$ on a family of Hilbert
spaces $(H_t)_{t>1}$, that are finite  multiples (see [GHJ], [Ra1]) of the left regular representation, taken with the cocycle corresponding  to the projective unitary representation. Thus we assume, using the notations from [GHJ], that 
$$\dim _{\{\pi_t(\Gamma)\}''}H_t=\chi_t,\  \  t>1.$$

 We also assume the existence two collections of bounded maps $j_{s,t}$, $\Delta_{s,t}:H_t\rightarrow H_s$, $s\geq t>1$,
verifying Chapman-Kolmogorov conditions for $s\geq t\geq r$.

 We assume that maps $\Delta_{s,t}$ intertwine the representations $\pi_s$ and $\pi_t$ of the group $\Gamma$.  In the particular case of Berezin's quantization, with the Hilbert spaces considered above, the maps $j_{s,t}$ are the 
obvious embeddings, while $\Delta_{s,t}$ are the multiplication operators  by $\Delta^{s-t}$.

 In addition, we assume that  the
following diagram is commutative, for $s\geq t>1$, $\varepsilon>0$:
$$
\begin{array}{ccc}
H_t & 
\mathop{\longrightarrow}\limits^{\Delta_{s,t}}& H_{s}  \\
\hskip-0.8cm{}_{j_{t,t-\varepsilon}} \Big\uparrow{}& &\hskip0.4cm\Big\uparrow{}_{j_{s,s-\varepsilon}} \\ [8pt]
H_{t-\varepsilon} & \mathop{\longrightarrow}\limits_{\Delta_{s-\varepsilon,t-\varepsilon}}& H_{s-\varepsilon} \\ 
\end{array}
$$
Furthermore, we assume that $j_{s,t}$, $\Delta_{s,t}$ are injective and that   $j_{s,t}$ has dense range.

 It automatically follows that  the orthogonal projection $E^s_t$ onto the
closure of the range of $\Delta_{s,t}$, belongs to $\A_s$, and its   Murray von Neumann dimension (trace) is equal to $\frac{\chi_t}{\chi_s}$.

Also, it automatically follows that the linear, densely defined maps $\Phi^t_\varepsilon$ from $H_t$ into $E^t_{t-\varepsilon}H_t$,
defined on $\Im j_{t,t-\varepsilon}$ by the formula
$$
\Delta_{t-\varepsilon,t}\circ j^{-1}_{t,t-\varepsilon}=j^{-1}_{t,t+\varepsilon}\circ\Delta_{t,t+\varepsilon}
$$
form a semigroup. We observe that the ranges of $j_{t-\varepsilon, t}$, $\Delta_{t-\varepsilon, t}$ are
increasing with $\varepsilon$, by the Chapman-Kolmogorov property. In particular $\chi_t$ is an increasing function of $t$.

We call a structure as above a "rigged"chain of projective, unitary representations of the group $\Gamma$.
\end{defn}

\

Given such a structure, we may define the maps $j_{t, t-\varepsilon}\cdot j^*_{t, t-\varepsilon}$
from $\A_{t-\varepsilon}$ onto $\A_t$, by mapping $a\in \A_{t-\varepsilon}$ into 
 $j_{t, t-\varepsilon}a j^*_{t, t-\varepsilon}\in \A_t$ for $1<t-\varepsilon$.
 
Note that in the case of the Berezin quantization, the generator of  $\Phi^t_\varepsilon$ is
$T^t_{\ln\Delta}$, the  unbounded Toeplitz operator on $H_t$ with analytic symbol $\ln\Delta$.

In this case, for a kernel $k$ representing an operator in $\A_t$, for $s\geq t$, the operator $j_{s,t} kj_{s,t}^*\in \A_s$ has symbol $\frac{\chi_s}{\chi_t}
(1-z\overline{\zeta})^{t-s} k(z, \zeta)$. 

Consequently the infinitesimal generator  $$\frac{d}{ds}\left[ j_{s,t} \cdot j_{s,t}^*\right]|_{s\rightarrow t_+}$$
 is exactly the Toeplitz  operator
$\T_{-\frac{1}{2} \frac{\chi'_t}{\chi_t}+ \ln(1-z\overline{\zeta})}$, that is constructed in the next section, in Definition \ref{gentop}.

The operators $j^*_{s,t} j_{s,t}$  are exactly the Toeplitz operators with  symbol $(1-|z|^2)^{s-t}$, acting on  on $H_t$. Likewise the operator  
$j_{s,t}$ $j^*_{s,t}$, acting  on $H_s$, is the inverse of the unbounded Toeplitz 
operator on $H_t$ with symbol $(1-|z|^2)^{t-s}$.

In particular, most of the properties of the domain of $\L_t$, can be read
from properties of the semigroup $\Phi^t_\varepsilon$ at the level of Hilbert
spaces, because of the following formula:
$$
\Phi^t_\varepsilon(a)=(j^{-1}_{t+\varepsilon, t}\circ \Delta_{t, t+\varepsilon}) a
(\Delta_{t, t+\varepsilon}^*\circ {(j^{-1}_{t, t+\varepsilon})^\ast}),\  a\in \A_t.
$$

For fixed $t>1$, consider the completely positive semigroup of operators $(\Theta^t_\varepsilon)_{\varepsilon \geq 0}$ acting on $B(H_t)$, defined by multiplying the reproducing kernel of an operator with the positive definite kernel on $\Bbb H^2$, which in this model is identified with the unit disk $\Bbb D$, defined on
$\Bbb D\times \Bbb D$ by:
$$(z,\eta) \rightarrow \frac {1} {(1-z\overline{\eta})^\varepsilon}, z,\eta \in \Bbb D.$$
Thus for a kernel $k$ representing an operator in $B(H_t)$, the formula for the kernel
of the operator $\Theta^t_\varepsilon(k) \in B(H_t)$ is
$$\Theta^t_\varepsilon(k)(z,\overline{\eta})=k(z,\overline{\eta})\frac {1} {(1-z\overline{\eta})^\varepsilon},\ \  z,\eta \in \Bbb D, \varepsilon \geq 0.$$
The positivity properties mentioned above were proved in the second chapter in [Ra2]. Note that $\Theta^t_\varepsilon$ is the Toeplitz operator introduced in Definition \ref{gentop}, with symbol 
$\frac{1}{(1-z\overline{\eta})^\varepsilon}$.

 Using this semigroup, it is easily observed that we may also reconstruct the symbol map. Indeed, for $s\geq t > 1$, we have the following equality:
$$\Psi_{s,t}(a)=\frac{\chi_s}{\chi_t}j_{s,t} \Theta^t_{s-t}(a) j_{s,t}^\ast,\ \  a\in \A_t.$$

\ 

\

\section{Higher order Berezin symbols}

In this section we use the machinery of Berezin quantization,
by using the symbols not just to define linear operators, but rather to define (eventually unbounded) linear
operators on spaces of operators. We will also  explain  how the model can be
applied to multilinear operators.

Recall, from last section, that any bounded operator $A$ is $B(H_t)$ 
is represented by a kernel $k_A: \bH\times\bH\rightarrow \bC$, in such
a way that
$$
Af(z)=\chi_t\int_{\bH}\frac{k_A(z,\overline{\eta})}{({z}-\overline{\eta})^t}
f(\eta)d\nu_t(\eta),\ \ 
z\in \bH, f\in H_t.$$
 Recall that $k_A$ is analytic in $z$ and antianalytic in
$\eta$. Here $\chi_t=\frac{t-1}{\pi}$.

The Hilbert space $\cC_2(H_t)$ of Hilbert-Schmidt operators induces 
the scalar product on the space of reproducing  kernels $k$, $l$,  representing operators $\cC_2(H_t)$ given by the formula
$$
\langle k,l \rangle=\langle k,l \rangle_{\cC_2(H_t)}=
\chi_t\iint_{\bH^2}k(z,\overline{\eta})\overline{l(z,\overline{\eta})}d_\Bbb H^t(z,\eta)
d\nu^2_0(z,\eta).
$$
For the space $L^2(\A_t)$, if $k, l\in L^2(\A_t)$, the scalar product formula is 
$$
\langle k,l \rangle=
\chi_t\iint_{F\times\bH}k(z,\overline{\eta})\overline{l(z,\overline{\eta})}d_{\Bbb H}^t(z,\eta)
d\nu^2_0(z,\eta),
$$
where $F$ is a fundamental domain for the action of the group $\Gamma$ on $\bH$.

\begin{defn}\label {toeplitz}
Note that both $\cC_2(H_t), L^2(\A_t)$ are spaces of functions
analytic in the first variable, antianalytic in the second.
Hence, for every measurable function $d$ on $\bH\times\bH$, we define an unbounded
Toeplitz operator (bounded if $d$ is bounded) by the~formula
$$
\langle \T_d^t k,l \rangle=
\chi_t\iint_{\bH^2}
d(z,\eta)k(z,\overline{\eta})\overline{l(z,\overline{\eta})}d^t_\bH(z,\eta) d^2\nu_0(z,\eta).
$$

In the case of $L^2(\A_t)$, if $d(z,\zeta)=d(\gamma z, \gamma\zeta)$, $\gamma\in\Gamma$,
$z, \zeta\in \bH$, by replacing, in the previous formula, the integration domain $\bH\times\bH$ by the
integration domain $F\times\bH$,  we obtain a Toeplitz operator $\T^{t,\Gamma}_d$ densely
defined on $L^2(\A_t)$. Recall that $F$ is a fundamental domain for the action of $\Gamma$ on $\bH$.
\end{defn}

\

Certainly, the two different types of Toeplitz operators are related,
and  we also define a third type of Toeplitz operator as 
follows:

\begin{defn}\label{gentop} Let $d:\bH\times\bH\rightarrow \bC$ be a measurable function.
We define a linear operator $Q_d=Q^t_d$ on a subspace $\D(Q_d)$  of $B(H_t)$, by the
pairing formula
$$
\langle Q_d(k), l \rangle = \Tr(Q_d( k)l^*)= \chi_t\iint_{\bH^2}d(z,\eta)k(z,\overline{\eta})\overline{l(z,\overline{\eta})}
d_\bH^t(z,\eta) d\nu^2_0(z,\eta).
$$

Here $l$ runs over the nuclear operators $\cC_1(H_t)$, and the domain $\D(Q_d)$ of $Q_d$
consists in all $k\in B(H_t)$, such that $Q_d(k)$ defined by the above formula is a 
bounded operator on $H_t$. So $Q_d$ is an unbounded operator on the Banach space $B(H_t)$.
\end{defn}

It is not necessary that $d$ bounded implies that $Q_d$ is 
bounded.
If $d$ is diagonally $\Gamma$-invariant, then
$Q_d\mid L^2(\A_t)\cap\D(Q_d)$ is the Toeplitz operator $\T^{t,\Gamma}_d$
introduced in the previous definition. Similarly $Q_d$, by restriction to $\mathcal C_2(H_t)$, is a Toeplitz operator,   if the domain intersection is non-void. 

Clearly,
$Q_{\overline{d}}$ restricts to the adjoint, but, apparently,  there is no functorial method of
constructing $Q_{\overline{d}}$, in an operator theoretic method, starting
from $Q_d$ (see more on this below).

Note that most of the above definitions could  also be introduced for distributions
$d(z,\zeta)$ acting on a dense subspace of the linear span of
functions of the form $k(z,\overline{\eta})\overline{ l(z, \overline{\eta})}$. In this case we 
will use the notation 
$$
\langle d(z,\zeta), k(z,\overline{\eta})\overline{l(\overline{z},\eta)} \rangle
$$
for the evaluation form associated to the distribution.

We collect below a number of elementary properties of these operators.

\begin{prop} We use below the notations from the previous two definitions.

{\rm 1)} If we restrict the operator $Q_d$ to $\cC_2(H_t)$ or $L^2(\A_t)$, then 
$Q_{\overline{d}}\mid\cC_2(H_t)$ or respectively 
$Q_{\overline{d}} \mid L^2(\A_t)$
is contained in the adjoint of $Q_d$.

{\rm 2)} Recall  that given a kernel $k$ representing the symbol of a bounded operator on $H_t$, the symbol  $k^*$  of the adjoint operator  is given by the formula
$k^*(z, \overline{\eta})=\overline{k(\eta, \overline{z})}$, 
$z, \eta\in \bH$.

Given $d$ as in the previous definitions, we introduce $s(d)(z,\zeta)=\overline{d(\zeta, z)}$
for $z, \zeta\in \bH$.

Let $C_d(k)$ be defined by the formula $C_d(k)=Q_d{(k^*)}^*$.
Assuming that the domain  $\D_{Q_d}$ of the operator $Q_d$  is closed under the adjoint operation, it 
follows that $C_d=Q_{s(d)}$.

{\rm 3)} If $d$ is a distribution as above, then the condition that $Q_d$ is
a derivation is that the measures
$$
\int_\bH d_{13}d\nu_t^2 {\   and\  } \int_\bH d_{12}d\nu_t^3+
\int_\bH d_{23}d\nu_t^1,
$$
coincide when evaluated on functions on $\bH^3$, in the variables
$z,\eta,\zeta\in \bH$, belonging to  the space 

$$\M=Sp\left\{\frac{k(z,\overline{\eta})}{(z-\overline{\eta})^t}\cdot
\frac{l(\eta,\overline{\zeta})}{(\eta-\overline{\zeta})^t}\cdot\frac{m(\zeta,\overline{z})}{(\zeta-\overline{z})^t}
\mid k, l, m\in \D({Q_d})\right\}.$$

By $d_{13}$ we mean the distribution $d$ acting on first and third
variable, and the integral is evaluated in the second variable. Similarly 
for the other two integrals$.)$ In the above formula, by integration with respect to the measure $d\nu_t^j$, $j=1,2,3$, we designate integration by the $j$-th variable from the sequence $z,\eta,\zeta\in \bH$.

In general, if $Q_d$ is a derivation it does not imply that 
$Q_{\overline{d}}$ is a derivation.
The above formula proves that this would happen if the spaces $\M$ and $\overline{\M}$ coincide  $(\overline{\M}$ is
the complex conjugate subspace of $\M$ $)$.

{\rm 4)} If $Q_d\mid L^2(\A_t)\cap \D({Q_d})$ is a derivation, and if $d$ is 
diagonally $\Gamma$-invariant $($so that the range is in $\A_t)$, then the
condition on $d$,  for $Q_d$ to be a derivation, is similar, the only difference being that we replace 
the space $\M$ by the space $\M_\Gamma$ of diagonally $\Gamma$-invariant functions on
$\bH$, obtained by requiring that $k, l, m$  correspond to diagonally 
$\Gamma$-invariant operators.

The $\Gamma$ invariance property for  functions $\Theta\in \M_\Gamma$, is
$$
\Theta(\gamma z,\gamma\zeta,\zeta\eta)=\Theta(z,\zeta,\eta)|j(\gamma,z)|^2|j(\gamma,\zeta)^2|\,
|j(\gamma,\eta)^2|
$$
for all $\gamma\in \Gamma$, $z, \zeta, \eta$ in $\bH$.

{\rm 5)} The conditions in  3), 4) above are verified if
$$
d(z, \zeta)=d_1(z)-d_1(\zeta),\ \  z,\eta \in \Bbb H,
$$
for some measurable function $d_1$ on $\bH$.
If $Q_d$ is real, that is $Q_d(k^*)=Q_d(k)^*$, for all $k$ in the domain, then 
$d$ takes only imaginary values.

In particular, $Q_{d}$ is antisymmetric in this case.

{\rm 6)} For $k$ in the domain of $Q_d$ the formula for 
$
\langle Q_d(k)v_1,v_2\rangle$, is
$$
\langle Q_d(k) v_1,v_2\rangle_{H_t}=\langle d(z,\zeta), k(z,\overline{\zeta})v_1(z)
\overline{v_2(\zeta)}\rangle,\ \  v_1, v_2\in H_t 
$$

{\rm 7)} If $d$ depends only on the first (or second) variable and is a 
measurable function, then
$Q_d(k)=T^t_d k$ $($respectively $Q_d(k)=kT^t_d)$ where $T^t_d$ is the $($unbounded$)$
Toeplitz operator acting on $H_t$ with symbol $d$.

{\rm 8)} Let $a(z,\zeta)=\arg(z-\overline{\zeta})$, $z, \zeta\in \bH$.
Then $(iQ_a)1=-\frac{1}{2}\frac{\chi'_t}{\chi_t} 1$.
However on $\cC_2(H_t)$, $iQ_a$ is antisymmetric.

\end{prop}

The properties $1)$, $2)$ are obvious. The property $3)$ is deduced
as follows: we start with the equality
$$
Q_d(kl)=kQ_d(l) + Q_d(l)k,\quad k,l\in \D_{Q_d}
$$
and pair this with $m\in \cC_1(H_t)$.
We obtain an identity of the form
$$
\langle Q_d(kl),m\rangle=\langle Q_d(l),k^* m\rangle+\langle Q_d(l),km^*\rangle
$$
which by writing explicitly the multiplication gives
the identity in the statement.

Property $4)$  follows similarly, using the trace on
$\A_t$ instead of the pairing.

Property $5)$ is a consequence of the definition,  taking 
$m\in\cC_1(H)$ to be the kernel associated to $\langle v_1,\cdot\,\rangle v_2$.

Property $6)$ is a consequence of Property $5)$. 

Property $7)$ is deduced
by differentiating in  $t$ the identity
$$
\chi_t\int_{\bH} \frac{1}{(z-\overline{\zeta})^t} d\nu_t=(\Im z)^t.
$$

We observe that some of the previous definitions may be extended to multilinear
operators. In particular, we have:

\begin{prop}
{\rm  1)} If $\Theta:\bH^{n+1}\rightarrow \bC$ is an Alexander
Spanier $n$-th cocycle on $\bH^{n+1}$ ([Gh]) and $\Theta$ is diagonally 
$\Gamma$-invariant,  then the following formula,  ($z,{\zeta}\in \bH$)
\begin{gather*}
c_\Theta(k_1,k_2,\ldots,k_n)(z,\overline{\zeta})=
\chi_t^{n+1}(z-\overline{\zeta})^t\\
\iint\ldots\int_{\bH^{n+1}}\frac{\Theta(\eta_1\ldots\eta_{n+1})k_1(\eta_1,\overline{\eta_2})
\ldots k_n(\eta_n,\overline{\eta_{n+1}})}
{(z-\overline{\eta_1})^t(\eta_1-\overline{\eta_2})^t\ldots
(\eta_{n+1}-\overline{\zeta})^t} d\nu_t^{n+1}(\eta_1\ldots\eta_{n+1})
\end{gather*}
defines a densely defined $n$-Hochschild cocycle on a dense
subspace $\D$ of $\A_t$. ($k_1,...,k_n\in \D)$.

{\rm  2)} In the case $n+1=3$, we obtain a densely defined Hochschild
2-cocycle
$$
c_\Theta(k,l)(z,\overline{\zeta})\!=\!\chi_t^3(z-\overline{\zeta})^t\iiint_{\bH^3}\frac{\Theta(\eta_1,\eta_2,\eta_3)k(\eta_1,\overline{\eta_2})
l(\eta_2,\overline{\eta_3})}{(z\!-\!\overline{\eta_1})^t(\eta_1\!-\!\overline{\eta_2})^t(\eta_2\!-\!\overline{\eta_3})^t
(\eta_3\!-\!\overline{\zeta})}d\nu^3_t(\eta_1,\eta_2,\eta_3).
$$

{\rm  3)} If $\Theta(\eta_1,\eta_2,\eta_3)= \phi(\eta_1,\eta_3)-\phi(\eta_1,
\eta_2)-\phi(\eta_2,\eta_3)$, for a measurable, $\Gamma$-diagonally
invariant function $\phi$ on $\bH$, then let
$$
X_\phi(k)(z, \overline{\zeta})=c^2_t\iint_{\bH^2}\frac{\phi(\eta_1,
\eta_2)k(\eta_1,\overline{\eta_2})}{(z-\overline{\eta_1})^t(\eta_1-
\overline{\eta_2})^t(\eta_2-\overline{\zeta})^t}d\nu_t^2(\eta_1,\eta_2).
$$
Here the domain consists in all $k$ such that the above integral defines an element in $\A_t$.
Then on the joint domain for $c_\Theta$, $X_\phi$ we have 
$$
c_\Theta(k_1,k_2)= X_\phi(k_1 k_2)-k_1 X_\phi(k_2)-X_\phi(k_1)k_2.
$$
Note that in fact for $k,l$ in the domain, we have 
$$
\langle X_\phi k, l\rangle_{L^2(\A_t)}=\chi_t\iint_{F\times\bH}\phi(\eta_1,\eta_2)k
(\eta_1,\overline{\eta_2})\overline{l(\overline{\eta_1},\eta_2)}d\nu_t(\eta_1,
\eta_2).
$$
Thus $X_\phi$ is in fact the Toeplitz operator $\T_\phi^{t,\Gamma}$ introduced in Definition \ref{toeplitz}.

{\rm  4)} Clearly we have the following equality: $$c_\Theta(k^*_n, k^*_{n-1},\ldots,k^*_1)=
{c_{\overline{\Theta}}(k_1, k_2,\ldots, k_n)}^*,$$
if the elements $k_1$, $\ldots,k_n$ belong to the domain of the corresponding cocycles and $\overline{\Theta}$ is the conjugate of the function $\Theta$.

{\rm  5)} In particular, if
$$
\Theta(\eta_1,\eta_2,\ldots,\eta_{n+1})= \Theta(\eta_2,\eta_3,\ldots,\eta_{n+1},
\eta_1)= \overline{\Theta(\overline{\eta_{n+1}},\overline{\eta_n},\ldots,
\overline{\eta_1})}
$$
then $\Psi_\Theta(k_1,\ldots,k_{n+1})=\tau_{\A_t}(c_\Theta
(k_1,\ldots,k_n)k_{n+1})$ is a cyclic cohomology cocycle defined on the domain of the corresponding cocycles.

\end{prop}

We particularize to the case of the cocycles that arise in the $\Gamma$-
equivariant Berezin quantization.

\begin{obs} Using a principal branch for the logarithmic function, let $a(z,\eta)$ be the argument of $(z-\overline{\eta})$, $z,\eta \in \Bbb H$.
 The cocycle $c^0_t$ introduced in the last 
chapter corresponds~to the Alexander Spanier cocycle
$$
\Theta_0(\eta_1,\eta_2,\eta_3)=-\frac{1}{2}\frac{\chi'_t}{\chi_t}+
i[a(\eta_1,\eta_3)-a(\eta_1,\eta_2)-a(\eta_2,\eta_3)].
$$
In particular, $c_t^0(k,l)=c_{\Theta_0}(k,l)$. Moreover, because of  3) in the previous statement, we have that
  $$c_t^0(k,l)=c_{\Theta_0}(k,l)=Z^t(kl)-kZ^t(l)-Z^t(k)l,$$
   for kernels $k,l$ representing operators in the domain of $c_t^0$.
Here $Z^t$ is the operator $$-\frac{1}{2}\frac{\chi'_t}{\chi_t}{\rm \ Id}+\T^{t,\Gamma}_{i\varphi_0}.$$
Recall  that $$\varphi_0(z,\zeta)=a(z,\zeta)+\arg\Delta(z)-
\arg\Delta(\zeta), z, \zeta\in \bH.$$

Here we are using a principal, analytic branch of the logarithm of the non-zero function $\Delta$, and by 
$\arg \Delta$ we denote the imaginary part of the logarithm.
\end{obs}

\begin{obs}
 The cocycle $c_t^0$ makes also sense as a 
densely defined cocycle on $B(H_t)$. It is the cocycle $\tilde{c_t^0}$
associated to the deformation consisting in the family of algebras $B(H_t)_{t>1}$.

In this case
$$
\tilde{c_t^0}(k, l)=\tilde{\Lambda^0_t}(k l)-\tilde{\Lambda^0_t}(k)l-k\tilde{\Lambda^0_t}(l),
$$
where $\tilde{\Lambda^0_t}=-\frac{1}{2}\frac{\chi'_t}{\chi_t}+iQ_a$, and 
the domain of $\tilde{\Lambda^0_t}$ and $\tilde{c_t^0}$  is the domain of $Q_a$. Here $Q_a$ is the unbounded operator introduced in Definition \ref {gentop}, and the function $a$ was defined in the previous definition.
Note that in this case $\tilde{\Lambda^0_t}(1)=0$, by Property $7)$,
from the previous statement.
\end{obs}

Recall that the modular factor for an element $\gamma \in \Gamma$, with
$$\gamma=\begin{pmatrix}
a& b \\ c & d \end{pmatrix},\ \ \  a,b,c,d\in \Bbb Z,$$
is $$j(\gamma, z)= (cz+d),\ \  z\in \bH.$$
Since $j(\gamma, z)\ne 0$, for all $z$, we choose a 
an analytic branch of the logarithm and taking the imaginary part we define $\arg(j(\gamma, z)), z\in \bH.$

\begin{cor} Any densely defined derivation $D$ on $B(H_t)$, or on
$\A_t$ that has the property that the unbounded operator $\tilde{Z_t}=D+\tilde{\Lambda^0_t}$ maps a dense $\ast$-subalgebra $\tilde {\D}$ of
$\A_t$ into $\A_t$ $($so that, on its domain intersected with $L^2(\A_t)$, the operator$\tilde{Z_t}$
implements the cocycle $c^0_t)$, is a solution to the $1$-cohomology {\rm ([Pe], [HaVa], [Sau])} problem
$$
\gamma D-D=-i[T^t_{\arg(j(\gamma,z)}, \cdot\,]=K(\gamma),\quad \gamma\in \Gamma.
$$

In the above formula, by $\gamma D$ we denote the derivation defined by the formula 
$(\gamma D)(k)= \pi_t(\gamma)D(k)\pi_t(\gamma)$,
$k$ in the domain~of~$D$. Here $T^t_{\arg(j(\gamma,z)}$ is the Toeplitz operator acting on $H_t$ with symbol  $\arg(j(\gamma,z)$.

Note that $K$ is $1$-cocycle for the group $\Gamma$, ([HV],[Pe]) with values in the bounded 
derivations on $B(H_t)$.
The condition that $K$ is 1-cocycle is that $$K(\gamma_1\gamma_2)=\gamma_1 K(\gamma_2)+K(\gamma_2),\gamma_1,\gamma_2\in\Gamma.$$
 We are therefore looking for a derivation that 
implements this $1$-cocycle.

\end{cor}

\begin{proof} 
This is simply a consequence of the fact that 
$\gamma\tilde{\Lambda_0^t}-\tilde{\Lambda_0^t}$
has the property that $\gamma\tilde{\Lambda_0^t}-\tilde{\Lambda_0^t}$ is a Toeplitz operator of the form introduced in Definition \ref{gentop}  corresponding to the 
unbounded symbol $i(\arg(j(\gamma,z))-\arg(j(\gamma,\zeta))$, $z,\zeta \in \bH$.
Consequently, the coboundary operators  for the Hochschild 2-cocycle
$c^0_t$ are in one to one correspondence with the derivations implementing the 1-cocycle $K$  on $\Gamma$, with values in bounded derivations.
\end{proof}

\begin{rem}
The canonical  solution $D_0$  of the 1-cocycle problem $\gamma D-D=K(\gamma)$
is in fact the derivation $D_0=[T^t_{i\arg\Delta},\cdot]$, where $T^t_{i\arg\Delta}$ is the Toeplitz operator with symbol $i\arg\Delta$ acting on $H_t$.
This corresponds to the canonical solution we constructed in 
Chapter $1$. Moreover, if $D$ is any other derivation, with the
same domain as $D_0$, then if $D=[T_{\sqrt{-1}d},\cdot]$, thus if $D$ is 
given as the form $Q_{i(d(z)-d(\zeta))}$ for some measurable real valued function 
$d$ on $\bH$, then the function on $\bH$ defined by $d(z)-\arg\Delta(z)$,
$z\in \bH$, is a $\Gamma$-invariant function,
and hence $T^t_d-T^t_{\arg\Delta(z)}$ is affiliated to $\A_t$.
\end{rem}

\begin{proof} The fact that $D_0$ is a solution to the 1-cocycle 
problem is simply a consequence of the fact that
$\arg\Delta(\gamma z)-\arg\Delta(z)=\arg(j(\gamma,z))$.

On the other hand, if $D$ is any other derivation on the same domain
as $[M_{\arg\Delta(z)},\cdot]$  that verifies the 1-cohomology
problem, and $D$ is of the form $T_{i(d(z)-d(\zeta))}$, then for any
$\ z, \zeta\in \Bbb H, \gamma \in \Gamma$, we have
$$d(\gamma z)-d(\gamma\zeta)-[d(z)-d(\zeta)]=\arg(j(\gamma,z))-\arg(j(\gamma,\zeta).$$ 
Because $\Gamma$ has no characters, this implies that $d(z)-\arg\Delta(z)$
is $\Gamma$-invariant.
\end{proof}

Finally, we recall  that the domain of the form $\L^0_t$, introduced in
Chapter~1,  is in fact the domain of the generator semigroup of unbounded
maps $\Phi^t_\varepsilon$ on $\A_t$.
Recall that $\Phi^t_\varepsilon(k)$ is the operator with symbol
$$
k(z,\zeta)\Delta(z)\overline{\Delta(\zeta)}{({(\overline{z}-
\zeta)}^{12})}^\varepsilon \frac{\chi_{t-\varepsilon}}{\chi_t},\quad \varepsilon>0.
$$
As proved in [Ra2], this generator, has a dense,  $\ast-$algebra domain $\D^0_t\subseteq L^2(\A_t)$.

 In fact our computation proves that this domain contains the intersection between the domain
$\D([T^t_{\ln\Delta},\cdot])\subseteq B(H_t)$ of the unbounded derivation $[T^t_{\ln\Delta},\cdot]$ and the domain $\D(\T^t_{\arg(1-z\overline{\zeta})})$ of the Toeplitz operator with symbol $\arg(1-z\overline{\zeta})$ considered in Definition \ref{gentop}.

We consider  the problem of finding derivations $Q$ as in Definition \ref{gentop}, densely
defined on a dense *-algebra domain $\D^0_t\subseteq \A_t$, so that $Q$ admits an extension to a derivation $\tilde{Q}$ with domain $\D(\tilde{Q})$
containing $\D([T^t_{\ln\Delta},\cdot])$, and taking value $0$ on $\A_t'$, that is $\tilde{Q}(k\pi_t(\gamma)=
\tilde{Q}(k)\pi_t(\gamma)$,  for $k\in \D(\tilde{Q})$,  $\gamma \in \Gamma$.

We first prove the following

\begin{prop}
Assume $Q=Q_d$  is a Toeplitz 
operator as in Definition \ref{gentop}, having as  symbol, a measurable function $d=d(z,\zeta)$ on
$\bH\times\bH$, i.e. 
$$
\langle Q_d k,m\rangle=\chi_t\iint_{\bH^2}{d(z,\zeta)k(z,\overline{\zeta})\overline{m(z,\overline{\zeta})}
d_\bH^t(z,\zeta)d\nu^2_0(z,\zeta)}.
$$
We assume that $Q$ admits a domain $\D$, that is  a dense $*$-subalgebra of
$B(H_t)$ so that $\D\cap \A_t$ is dense in $L^2(\A_t)$, and such that 
$Sp\{\pi_t(\gamma)\mid \gamma\in \Gamma\}$ is contained in $\D$.

Assume that $Q(k\pi_t(\gamma))=Q(k)\pi_t(\gamma)$ and that  $\D$ is sufficiently large so that the span of $k(z,\overline{\zeta})\overline
{m(z,\overline{\zeta})}$, $k\in \D$, $m\in \cC_1(H_t)$ is dense in
the continuous functions $C(\bH\times\bH)$.

 Then $d$ is ($\Gamma\times\Gamma$)-invariant. In particular,
$Q_d(\D\cap\A_t)\subseteq\A_t$.
\end{prop}

\begin{proof} We use this for symbols $m(z,\zeta)$ of the form $\overline{v_1(\zeta)}v_2(z)$,
where $v_1$, $v_2$ are vectors in $H_t$.
We have 
$$\langle Q_d(k)v_1, v_2\rangle=\iint_{\bH^2}{d(z,\zeta)k(z,\overline{\zeta})v_1(\zeta) \overline
{v_2(z)}d_\bH^t(z,\zeta)d\nu^2_0(z,\zeta)}.
$$
Then the identity $Q_d(k\pi_t(\gamma))=Q_d(k)\pi_t(\gamma)$ implies 
$$
\langle Q_d(k\pi_t(\gamma^{-1}))\pi_t(\gamma)v_1,v_2\rangle
=\langle Q_d(k)v_1,v_2\rangle$$
and hence
\begin{gather*}
\iint_{\bH^2}
d(z,\zeta)k(z,\overline{\gamma\zeta})v_1(\gamma^{-1}\zeta)\overline
{v_2(z)}
d_\bH^t(z,\zeta)d\nu^2_0(z,\zeta) =\\
=\iint_{\bH^2}d(z,\zeta)k(z,\overline
{\zeta})v_1(\zeta)\overline
{v_2(z)}d_\bH^t(z,\zeta)d\nu^2_0(z,\zeta).
\end{gather*}
We obtain:
$$
d(z,\gamma\zeta)=d(z,\zeta) 
$$
for all $z, \zeta$ in $\bH$, $\gamma\in \Gamma$.
\end{proof}
We  prove that for a large class of derivations $\delta$,
densely defined on $\A_t$, we can not perturb by $\delta$ the operator
$\lambda_t1+\L^0_t$, so as to get rid of the real part.

\begin{thm}\label{obstruction}
Let $c_t^0$ the unbounded Hochschild cocycle constructed in Chapter 1 ([Ra2]), associated to the $\Gamma-$ invariant Berezin quantization deformation. This cocycle is 
densely defined on the domain $\D^0_t\times \D^0_t $, where $ \D^0_t $ is the
dense $\ast$-subalgebra of $\A_t$ constructed in Corollary 6.6 in [Ra2].

Recall that $c_t^0$ admits a dense coboundary operator $Z_t=\lambda_t1+\L^0_t$ with domain $\D^0_t$, where $\lambda_t>0$ and
$\L_t^0$ an antisymmetric operator with $(\L^0_t)^*1=-\lambda_t 1$ 
(so $\L_t^0$  is not defined at $1$). 

Let $\delta$ be a  densely defined derivation on $\D^0_t\subseteq L^2(\A_t)$, and so that $\delta$
admits an extension $\tilde Q$ to a domain $\tilde{\D}\subseteq B(H_t)$, verifying the conditions of the previous lemma.

 Then the operator $Z_t+\delta$
is not defined at $1$.
\end{thm}

\begin{proof} Because of the previous lemma, $\delta$ is of the form $\T_d^{t,\Gamma}$ with $d$ a function on
$\bH\times\bH$ that is $(\Gamma\times\Gamma)$-invariant.

We let $F$ be the canonical  fundamental domain  for the action of $\Gamma$ on $\bH$. We select a function, denoted $\arg_0\Delta(z)$, which we assume to be 
 a measurable choice of the argument of $\Delta(z)$,  taking values interval $(-\pi,\pi]$ and so that the function $\arg\Delta(z)- \arg_0\Delta(z), z\in \bH$, is almost everywhere constant, on the  translates of $F$ by elements in the group $\Gamma$.
 
  Then, using  the 
identification  of the Hilbert space $L^2(\A_t)$ with the Hilbert space $H^2(F\times\bH, d_\bH^t {\rm d}\nu^2_0)$, we observe that the bilinear form associated to the operator $\T^{t,\Gamma}_\varphi$, where $$\varphi(z,\zeta)=\arg(z-\overline{\zeta})+\arg\Delta(z)-
\arg\Delta(\zeta),$$ breaks into the bounded linear form
$$
\iint_{F\times\bH}{\left[\arg_0\Delta(z)-\arg_0\Delta(\zeta)+\arg(z-\overline{\zeta})\right]
k(z,\overline{\zeta})\overline{l(\overline{z},\zeta})d_\bH^t(z,\zeta)d\nu^2_0(z,\zeta)}
$$
and the unbounded bilinear form
$$
\iint_{F\times\bH}{\Big(\sum \theta_\gamma \chi_{\gamma F}(z)- \sum \theta_\gamma \chi_
{\gamma F}(\zeta)}\Big)k(z,\overline{\zeta})\overline{l(\zeta,\overline{z})}d_\bH^t(z,\zeta)d\nu^2_0(z,\zeta).
$$
In this expression, we have $\theta_\gamma= \arg\Delta(z)-\arg_0\Delta(z)$, for $z\in \gamma F$, $\gamma \in \Gamma$. Hence
$$\sum \theta_\gamma \chi_{\gamma F}(z)= \arg\Delta(z)- \arg_0\Delta(z),\ \  z\in \bH.$$

Hence if we consider the unbounded  bilinear form associated to 
$$\langle(Q_d+\L_t^0)k,l\rangle_{L^2(\A_t)},$$ then the unbounded part is 
$$
\iint_{F\times \bH} {\Big(\sum \theta_\gamma \chi_{\gamma F}(\zeta)-d(z,\zeta)\Big)
k(z,\overline{\zeta})\overline{l(\zeta,\overline{z})}d_\bH^t(z,\zeta)d\nu^2_0(z,\zeta)}.
$$
But $\theta_\gamma$ is the unique (unbounded) 1 cocycle coboundary  of the Euler cocycle, which is 
rapidly increasing to $\infty$, and $d$ is $(\Gamma\times\Gamma)$-invariant,
and hence this expression can not  contain the element $k=1$ in the domain of its form closure.
\end{proof}

We finally note that on the von Neumann algebra $L(F_\infty)$, associated to the free group $F_\infty$ with infinitely many generators, there exists a closed derivation $\delta$, 
so that $\Re\delta=\alpha, \alpha>0$. In addition there is an abelian algebra
$L^\infty([0, 1])$ invariated by $\delta$, and so that $\Im\delta$ has deficiency indices (0, 1) when restricted  to the Hilbert space of the smaller algebra, and has the same indices on the Hilbert space associated to  the full algebra.

\begin{prop}\label{infinite}We use the description of $L(F_\infty)$  as the infinite corner 
$\chi_{[0, 1]}(L^\infty([0,\infty])
\star\bC(X))\chi_{[0, 1]}$ introduced in {\rm [Ra3]}.

Here $X$ is an ``infinite semicircular" element (see {\rm [Ra3]}), and $\alpha_t$
acts on $L^\infty([0, \infty))\star \bC(X)$ by dilation on $L^\infty([0, \infty])$ 
and multiplying $X$ by $t^{-1/2}$.

Let $\D\subseteq \chi_{[0, 1]}(L^\infty([0, \infty))\star \bC(X))\chi_{[0, 1]}$ 
be the domain generated by 
the span of $f_0 X f_1 X\ldots Xf_{n-1}Xf_n$, 
where $f_0, f_n\in C_0^1([0, 1))$ and $f_1,\ldots, f_n$ 
are differentiable with compact support.

Then $\delta=\frac{d}{d_t}\chi_{[0, 1]}\alpha_t(x)\chi_{[0, 1]}$ is a 
derivation on $\D$, closable, with
$\delta^*1=\alpha>0$.
Moreover $\delta$ invariates a dense domain in $L^\infty([0, 1])$.
\end{prop}

\begin{proof} This is simply a consequence of the fact that $\alpha_t$ is scaling 
the trace, and acts as a dilation on $L^\infty([0, 1])$.
\end{proof}

\

  {\bf Acknowledgement}. The author is indebted to Professors Ryszard Nest and Pierre Bieliavsky for several comments during the elaboration of this paper. The author thanks P. Bieliavsky for providing him with a deep insight into his forthcoming paper ([Bi]). The author is grateful to the referee for pointing out to him several inaccuracies in a first version of this paper.

\end{document}